\documentclass[11pt,reqno]{amsart}
\usepackage{amsmath,amssymb}
\usepackage{cite}

 \makeatletter
 \evensidemargin  \oddsidemargin
 \makeatother

\newtheorem{theorem}{Theorem}[section]
\newtheorem{lemma}[theorem]{Lemma}
\newtheorem{proposition}[theorem]{Proposition}

\theoremstyle{definition}
\newtheorem{definition}[theorem]{Definition}
\newtheorem{example}[theorem]{Example}

\theoremstyle{remark}

\numberwithin{equation}{section}

 \newtheorem*{theorem*}{Theorem}

 \numberwithin{equation}{section}
 \numberwithin{equation}{section}
\makeatletter
\@namedef{subjclassname@2020}{%
  \textup{2020} Mathematics Subject Classification}
\makeatother

\begin{document}
\title[An extension of the weighted geometric mean]{An extension of the weighted geometric mean in unital $JB$-algebras}
\author[A. G. Ghazanfari]{A. G. Ghazanfari$^{1}$ and S. Malekinejad$^{2,*}$ }

 \setcounter{section}{0}
 \numberwithin{equation}{section}

\address{$^{1}$Department of Mathematics\\
Lorestan University, Khoramabad, Iran.}

\address{$^{2}$Department of Mathematics\\
Payame Noor University, Tehran, Iran.}

\thanks{*Corresponding author}

\email{$^{1}$ghazanfari.a@lu.ac.ir, $^{2}$maleki60313@pnu.ac.ir }

\subjclass[2020]{46H70, 47A63, 47A64}

\keywords{$JB$-algebra, Heinz mean, Heron mean}

\begin{abstract}\noindent
Let $\mathcal{A}$ be a unital $JB$-algebra and $A,~B\in\mathcal{A}$, we extend the weighted geometric mean $A\sharp_r B$, from $r\in [0,1]$ to $r\in (-1, 0)\cup(1, 2)$.
We will notice that many results will be reversed when the domain of $r$ change from $[0,1]$ to $(-1,0)$ or $(1,2)$.
We investigate some property of $A\sharp_r B$ for such quantities of $r$, such as
 we show that $A\sharp_r B$ is separately operator convex with respect to $A, B$.
We also introduce the Heinz and Heron means for unital $JB$-algebras and
give some famous inequalities involving them.
\end{abstract}
\maketitle

\section{introduction and preliminary}\vspace{.2cm} \noindent

Jordan algebras are considered as a model to formalize the concept of an algebra of observables in quantum mechanics.
In the mathematical foundation of quantum physic one of the natural axioms
 is that the observables form a jordan algebra.
One reason for this is that, typically, observables are supposed to
be self-adjoint operators on a Hilbert space, and the space of such operators
is closed under the Jordan product but not the associative product. We refer the reader to \cite{emc} for further details.

A Jordan algebra over $\mathbb{R}$ is a vector space $\mathcal{A}$ over $\mathbb{R}$
equipped with a commutative bilinear product $\circ$ that satisfies the identity
\begin{align*}
a\circ(b\circ a^2)=(a\circ b)\circ a^2,
\end{align*}
for all $a,b\in\mathcal{A}$.

Let $\mathcal{A}$ be an algebra and $a,b\in\mathcal{A}$. Let
\begin{align}\label{1}
x\circ y=\frac{x y+y x}{2}.
\end{align}
Then $\circ$ defines a bilinear, commutative product on $\mathcal{A}$, which is called the Jordan product.
If $\mathcal{A}$ is associative,
then  $\mathcal{A}$ becomes a Jordan
algebra when equipped with the product \eqref{1}, as does any subspace closed
under $\circ$.
Such Jordan algebras are called special Jordan algebras, all others are called exceptional.
The following algebras are examples of special Jordan algebras with product \eqref{1}.
 \begin{example}
 \end{example}
 \begin{itemize}
\item
  The Jordan algebra of $n\times n$ self-adjoint real matrices $H_n(\mathbb{R})$.
  \item
  The Jordan algebra of $n\times n$ self-adjoint complex  matrices $H_n(\mathbb{C})$.
\item
 The Jordan algebra of $n\times n$ self-adjoint quaternionic matrices $H_n(\mathbb{H})$.
\item The Jordan algebra of $n\times n$ self-adjoint octonionic matrices $H_n(\mathbb{O})$, where $n\leq3$.
\end{itemize}


\begin{definition}
A Jordan Banach algebra is a real Jordan algebra $\mathcal{A}$ equipped with a complete norm satisfying
$$\|A\circ B\|\leq \|A\|\|B\|,~~~~A,B\in\mathcal{A}.$$
Jordan operator algebras are norm-closed spaces of operators on a Hilbert space which are closed under the Jordan product.

Basic examples are real symmetric and complex hermitian matrices with the Jordan product.
A $JB$-algebra is a Jordan Banach algebra $\mathcal{A}$ in which the norm satisfies the following two additional conditions for $A,B\in\mathcal{A}$:
\begin{align*}
(i&)~\|A^2\|=\|A\|^2\\
(ii&)\|A^2\|\leq\|A^2+B^2\|
\end{align*}
\end{definition}
We say $A\in \mathcal{A}$ is invertible if there exists $B\in \mathcal{A}$, which is called Jordan inverse of
$A$, such that
$$A \circ B = I~~~\text{and}~~A^2 \circ B = A.$$
The spectrum of $A$, denoted by $Sp(A)$, is the set of $\lambda\in \mathbb{R}$  such that $A -\lambda $ does not
have inverse in $A$ . Furthermore, if $Sp(A)\subset[0,\infty)$, we say $A$ is positive, denoted
$A \geq 0$.\\
In a $JB$-algebra we define
\begin{equation*}
U_AB = \{ABA\} := 2(A \circ B) \circ A - A^2 \circ B.
\end{equation*}
Note that $ABA$ is meaningless unless $A$ is special, in which case $\{ABA\} = ABA$.
Moreover, if $B \geq 0$, then $U_AB =\{ABA\}\geq 0$.

We mention some of properties of $U_A$ that we will use frequently in sequel:
$U_A$ is a linear mapping and
\begin{align}\label{2}
 U_{\{ABA\}}=U_AU_BU_A.
\end{align}
It also satisfies the following two Lemmas:
\begin{lemma}\cite[Lemma 1.23]{1}\label{l1}
Let $\mathcal{A}$ be a $JB$-Banach algebra and $A\in\mathcal{A}$. Then $A$ is an invertible element iff $U_A$
has a bounded inverse, and in this case the inverse map is $U_{a^{-1}}$ i.e., $U^{-1}_{A}=U_{A^{-1}}$.
\end{lemma}
\begin{lemma}\cite[Lemma 1.24]{1}\label{l2}
If $A$ and $B$ are invertible elements of a JB-algebra, Then $\{ABA\}$ is invertible with inverse $\{A^{-1}B^{-1}A^{-1}\}$.
\end{lemma}
For more details, we refer the reader to \cite{1, han}.

A real-valued function f on $\mathbb{R}$ is said to be operator monotone (increasing) on a
JB-algebra $\mathcal{A}$ if $A\leq B$ implies $f(A) \leq f(B)$. We say $f$ is operator convex if for
any $\lambda\in[0, 1]$,
$$f((1-\lambda)A +\lambda B) \leq (1-\lambda)f(A) +\lambda f(B).$$

Wang et al. \cite{wang} introduced some operator means for two positive invertible elements $A$, $B$ in a unital $JB$-algebra $\mathcal{A}$
and $\nu\in[0,1]$, such as
\begin{itemize}
\item $\nu$-weighted harmonic mean: $A!_{\nu}B=\left((1-\nu) A^{-1}+\nu B^{-1}\right)^{-1}$;
\item  $\nu$-weighted geometric mean: $A\sharp_{\nu} B=\{A^{1/2}\{A^{-1/2}BA^{-1/2}\}^{\nu}A^{1/2}\}$;
\item $\nu$-weighted arithmetic mean: $A\nabla_{\nu} B=(1-\nu) A+\nu B$.
 \end{itemize}
The following relations among them are also proved in \cite{wang}.
\begin{align}
&A\sharp_{\nu}B=B\sharp_{1-\nu}A, \label{01}\\
&(A\sharp_\nu B)^{-1}=A^{-1}\sharp_\nu B^{-1}, \label{02}\\
&A!_{\nu}B\leq A\sharp_{\nu}B\leq A\nabla_{\nu} B, \label{03}\\
&(\alpha A\sharp_{\nu} \beta B)=(\alpha\sharp_{\nu}\beta)( A\sharp_{\nu}B) &&(\alpha>0, \beta>0)\label{3}\\
&\{C(A\sharp_\nu B)C\}=\{CAC\}\sharp_\nu\{CBC\}&&\text{ for any invertible } C \in \mathcal{A}.\label{04}
\end{align}

They also stated the following definition and proved the following lemmas.

\begin{definition}\cite[Definition 8]{wang}\label{d1}
Let $f_{\alpha}(x)$ be a real-valued function on $(\alpha,x)\in (0,\infty)\times[0,\infty)$ that is separately continuous with respect to $\alpha$ and $x$. We say $f_{\alpha}(x)$ is uniformly Riemann integrable on $\alpha\in(0,\infty)$ for $x$ on bounded  and closed intervals if for any closed interval $[0,M]\subset[0,\infty)$, and $\varepsilon>0$, there exist $1>\delta>0$ and $N>1$, such that for any positive numbers $\delta_1,\delta_2\leq\delta$, $N_1,N_2\geq N$, partitions $\Delta_{\beta}$,  $\Delta_{\gamma}$ of $[\delta,N]$ with $\max_{k,l}\{\mid\Delta\beta_k\mid,\mid\Delta\gamma_l\mid\}<\delta$, and for all $x\in[0,M]$, we have
\begin{align*}
&\left|\int_{\delta_1}^{\delta_2}f_{\alpha}(x)d\alpha\right|<\dfrac{\varepsilon}{3},~~\left|\int_{N_1}^{N_2}f_{\alpha}(x)d\alpha\right|<\dfrac{\varepsilon}{3},\\
&\left|\sum_k f_{\beta^*_k}(x)\Delta\beta_k-\sum_l f_{\gamma^*_l}(x)\Delta\gamma_l\right|<\dfrac{\varepsilon}{3},
\end{align*}
where $\beta^*_k\in[\beta_{k-1},\beta_k]$ and $\gamma_l^*\in[\gamma_{l-1},\gamma_l]$ are arbitrary.
\end{definition}
\begin{lemma}\cite[Lemma 1 \and 2]{wang}\label{l3}
Let $f_{\alpha}(x)$ be a family of (operator monotone) operator concave functions on a unital $JB$-algebra $\mathcal{A}$ indexed by $\alpha$ in $(0,\infty)$. Assume $f_{\alpha}(x)$ is uniformly Riemann integrable on $\alpha\in(0,\infty)$ for $x$ on bounded and closed intervals. Then $\int_0^{\infty}f_{\alpha}(x)d\alpha$ is also (operator monotone) operator concave.
 \end{lemma}
With an argument similar to the one used in the proof of the above lemma,
it can be shown that the above lemma is also valid for operator convex functions.

Bedrani et al., in \cite{bed1} introduced the geometric mean $A\sharp_rB$ for two accretive
matrices $A, B$, when $r \in (1, 2)$ and $r \in (-1, 0)$. They showed that many
results will be reversed when the domain of $r$ changes from $[0, 1]$ to $(-1, 0$) or $(1, 2)$.

The main goal of this paper is to extend the weighted geometric mean $A\sharp_{r} B$ for $r\in(-1,0)\cup (1,2)$ in the setting of $JB$-algebras.
We introduce the notion of $A\sharp_{r} B$, as follows
\begin{align}\label{4}
A\sharp_rB:=\{A^{\frac{1}{2}}\{A^{-\frac{1}{2}}BA^{-\frac{1}{2}}\}^rA^{\frac{1}{2}}\}
\end{align}
when $r\in(-1,0)\cup (1,2)$ and investigate some it's properties.
We also introduce the Heinz means, Heron means and give some famous inequalities involving them for $JB$-algebras.

\section{An extension of the weighted geometric mean}
First we state the following efficient Lemma.
\begin{lemma}\label{3.3}
Let $\mathcal{A}$ be a $JB$-algebra and $A\in \mathcal{A}$ and $r,s\in \mathbb{R}$. Then
\item[(i)]
$A^r\circ A^s=A^{r+s}~~~(A\geq 0)$,
\item[(ii)]
$U_AU_A=U_{A^2}$.
\end{lemma}

\begin{proof}
(i) Consider continuous real functions $f(t)=t^r,~g(t)=t^s$ on $(0,\infty)$. By the
continuous functional calculus at $A$, we have
\begin{equation*}
A^{r+s}=(fg)(A)=f(A)\circ g(A)=A^r\circ A^s.
\end{equation*}

(ii) For every $B\in \mathcal{A}$, The identity $$U_AU_A(B)=\{A\{ABA\}A\}=\{A^2BA^2\}=U_{A^2}(B)$$ follows from MacDonald's theorem.
\end{proof}

It is well known that the following integral representation is valid for $x>0$ and $0<r<1$. (See relation (V.4) in \cite{bah1}.)
\begin{align}\label{1.7}
x^r=\frac{\sin r\pi}{\pi}\int_0^\infty \frac{x}{\lambda+x}\lambda^{r-1} d\lambda.
\end{align}

For $x\in (0,\infty)$, from \eqref{1.7}, we obtain
\begin{align}
x^r&=\frac{\sin (r-1)\pi}{\pi}\int_0^\infty \frac{x^2}{\lambda+x}\lambda^{r-2} d\lambda && r\in(1,2)\label{1.8}\\
x^r&=\frac{\sin (r+1)\pi}{\pi}\int_0^\infty \frac{1}{\lambda+x}\lambda^{r} d\lambda && r\in(-1,0)\label{1.9}
\end{align}

By changing of variable $\lambda=\frac{1}{\alpha}$ from \eqref{1.7}, we have

\begin{align}\label{1.10}
x^r=\frac{\sin r\pi}{\pi}\int_0^\infty \frac{x}{1+\alpha x}\alpha^{-r} d\alpha.
\end{align}

Thus for $x\in (0,\infty)$, from \eqref{1.10}, we get
\begin{align}
x^r&=\frac{\sin (r-1)\pi}{\pi}\int_0^\infty \frac{x^2}{1+\alpha x}\alpha^{-r+1} d\alpha && r\in(1,2)\label{1.11}\\
x^r&=\frac{\sin (r+1)\pi}{\pi}\int_0^\infty \frac{1}{1+\alpha x}\alpha^{-r-1} d\alpha && r\in(-1,0)\label{1.12}
\end{align}

\begin{proposition}\label{p1}
Let $r\in[-1,0]\cup[1,2]$ and $x\geq0$, the real-valued function $x\rightarrow x^r$ is operator convex on unital JB-algebras.
\end{proposition}

\begin{proof}
Let $\mathcal{A}$ be a unital $JB$-algebra, $r\in(-1,0)$ and $x\geq0$.
From \eqref{1.12}, we have
\begin{align*}
\left|\int_{\delta_1}^{\delta_2}\frac{1}{1+\alpha x}\alpha^{-r-1}d\alpha\right|\leq
\left|\int_{\delta_1}^{\delta_2}\alpha^{-r-1}d\alpha\right|\rightarrow 0\text{ uniformly },
\text{ as }\delta_1,\delta_2\rightarrow 0.
\end{align*}

From \eqref{1.9}, we have
\begin{align*}
\left|\int_{N_1}^{N_2}\frac{\alpha}{\alpha+x}\alpha^{r-1}d\alpha\right|\leq\left|\int_{N_1}^{N_2}\alpha^{r-1}d\alpha\right|\rightarrow 0\text{ uniformly },
\text{ as }N_1,N_2\rightarrow\infty
\end{align*}
If $[a,b]$ is an arbitrary interval of $(0,\infty)$ and $\alpha\in [a,b]$, then

\begin{align*}
\sup_{x\in[0,M]}&\left|\frac{1}{(1+\alpha x)\alpha^{r+1}}-\frac{1}{(1+\alpha_0 x)\alpha_0^{r+1}} \right|\\
&\leq \frac{1}{a^{2(r+1)}}|(1+\alpha x)\alpha^{r+1}-(1+\alpha_0 x)\alpha_0^{r+1}|\rightarrow 0, ~ \text {as }\alpha\rightarrow\alpha_0.
\end{align*}

Hence $f_\alpha(x)=\frac{1}{1+\alpha x}\alpha^{-r-1}$ is uniformly Reimann integrable on $\alpha\in(0,\infty)$.
In this case, operator convexity the function $x\rightarrow x^r$ on $\mathcal{A}$, follows from
 Lemma \eqref{l3} that is also valid for operator convex functions.

Now if $r\in(1,2)$ and $x\geq0$.
From \eqref{1.11}, we have
\begin{align*}
\left|\int_{\delta_1}^{\delta_2}\frac{x^2}{1+\alpha x}\alpha^{-r+1}d\alpha\right|\leq
M^2\left|\int_{\delta_1}^{\delta_2}\alpha^{-r+1}d\alpha\right|\rightarrow 0\text{ uniformly },
\text{ as }\delta_1,\delta_2\rightarrow 0.
\end{align*}

Also
\begin{align*}
\left|\int_{N_1}^{N_2}\frac{\alpha x}{1+\alpha x}x\alpha^{-r}d\alpha\right|\leq M\left|\int_{N_1}^{N_2}\alpha^{-r}d\alpha\right|\rightarrow 0\text{ uniformly },
\text{ as }N_1,N_2\rightarrow\infty
\end{align*}
If $[a,b]$ is an arbitrary interval of $(0,\infty)$ and $\alpha\in [a,b]$, then

\begin{align*}
\sup_{x\in[0,M]}&\left|\frac{x^2}{(1+\alpha x)\alpha^{1-r}}-\frac{x^2}{(1+\alpha_0 x)\alpha_0^{1-r}} \right|\\
&\leq \frac{M^2}{a^{2(1-r)}}|(1+\alpha x)\alpha^{1-r}-(1+\alpha_0 x)\alpha_0^{1-r}|\rightarrow 0,~ \text {as }\alpha\rightarrow\alpha_0.
\end{align*}

Hence $g_\alpha(x)=\frac{x^2}{1+\alpha x}\alpha^{-r+1}$ is uniformly Reimann integrable on $\alpha\in(0,\infty)$.
In this case, operator convexity the function $x\rightarrow x^r$ on $\mathcal{A}$, follows from
 Lemma \eqref{l3} that is also valid for operator convex functions.

\end{proof}

In the next Theorems, for $r\in(-1,0)$ or $r\in(1,2)$, we give an integral representation of $A\sharp_{r} B$ in unital $JB$-algebras.
\begin{theorem}\label{t1}
Let $A$, $B$ be positive invertible elements in a unital JB-algebra $\mathcal{A}$. Then for any $r\in(1,2)$
\begin{align*}
&A\sharp_rB=\int_0^1\left((1-s)B^{-1}+s\{B^{-1}AB^{-1}\}\right)^{-1}d\mu(s),\\
&\text {where}~~d\mu(s)=\frac{\sin(r-1)\pi}{\pi}\frac{s^{r-2}}{(1-s)^{r-1}}ds.
\end{align*}
\end{theorem}

\begin{proof}
By a change of variables to \eqref{1.7}, we obtain
\begin{align}\label{2.a}
x^r=\frac{\sin r\pi}{\pi}\int_0^1 \frac{x}{s+(1-s)x}\frac{s^{r-1}}{(1-s)^{r}}ds,
\end{align}

Therefore for $r\in(1,2)$ and $x\in(0,\infty)$,
\begin{align}\label{8}
&x^r=\int_0^1x^2\left(s+(1-s)x\right)^{-1}d\mu(s),\\
&\text {where}~~d\mu(s)=\frac{\sin(r-1)\pi}{\pi}\frac{s^{r-2}}{(1-s)^{r-1}}ds.\notag
\end{align}
Applying functional calculus in $JB$-algebras to $(\ref{8})$, (\cite[Proposition 1.21]{1}).
We obtain
\begin{align*}
&\{A^{-\frac{1}{2}}BA^{-\frac{1}{2}}\}^r\\
&=\int_0^1\{A^{-\frac{1}{2}}BA^{-\frac{1}{2}}\}
\left(sI+(1-s)\{A^{-\frac{1}{2}}BA^{-\frac{1}{2}}\}\right)^{-1}\{A^{-\frac{1}{2}}BA^{-\frac{1}{2}}\}d\mu(s)\\
&=\int_0^1U_{\{A^{-\frac{1}{2}}BA^{-\frac{1}{2}}\}}
\left(sI+(1-s\right)U_{A^{-\frac{1}{2}}}B)^{-1}
d\mu(s).
\end{align*}
Therefore by \eqref{2}, \eqref{4}, Lemma $\ref{l1}$ and Lemma $\ref{l2}$, we get
\begin{align*}
&A\sharp_rB=\{A^{\frac{1}{2}}\{A^{-\frac{1}{2}}BA^{-\frac{1}{2}}\}^rA^{\frac{1}{2}}\}\\
&=\left\{A^{\frac{1}{2}}\int_0^1U_{\{A^{-\frac{1}{2}}BA^{-\frac{1}{2}}\}}
\left(sI+(1-s)U_{A^{-\frac{1}{2}}}(B)\right)^{-1}
d\mu(s)A^{\frac{1}{2}}\right\}
\\
&=\int_0^1U_{A^{\frac{1}{2}}}U_{A^{-\frac{1}{2}}}U_{B}
U_{A^{-\frac{1}{2}}}
\left(sI+(1-s)U_{A^{-\frac{1}{2}}}(B)\right)^{-1}d\mu(s)\\
&=\int_0^1U_{B}U_{A^{-\frac{1}{2}}}
\left(sI+(1-s)U_{A^{-\frac{1}{2}}}(B)\right)^{-1}d\mu(s)\\
&=\int_0^1\left((1-s)B^{-1}+s
\{B^{-1}AB^{-1}\}\right)^{-1}d\mu(s).
\end{align*}
\end{proof}
\begin{theorem}\label{t2}
Let $A$, $B$ be positive invertible elements in a unital JB-algebra $\mathcal{A}$. Then for any $r\in(-1,0)$

\begin{align*}
A\sharp_rB=\int_0^1\left((1-s)\{A^{-1}BA^{-1}\}+sA^{-1}\right)^{-1}d\nu(s),
\end{align*}
where
$d\nu(s)=\frac{\sin(r+1)\pi}{\pi}\frac{s^r}{(1-s)^{r+1}}ds$.

\begin{proof}
By \eqref{2.a} for $r\in(-1,0)$, we have

\begin{align}\label{2.0.0}
&x^r=\int_0^1\left(s+(1-s)x\right)^{-1}d\mu(s),\\
&\text {where}~~d\mu(s)=\frac{\sin(r+1)\pi}{\pi}\frac{s^{r}}{(1-s)^{r+1}}ds.\notag
\end{align}
Applying functional calculus in $JB$-algebras to $(\ref{2.0.0})$, (\cite[Proposition 1.21]{1}).
We obtain
\begin{align*}
A\sharp_rB&=\int_0^1\left((1-s)\{A^{-\frac{1}{2}}\{A^{-\frac{1}{2}}BA^{-\frac{1}{2}}\}A^{-\frac{1}{2}}\}+sA^{-1}\right)^{-1}d\mu(s)\\
&=\int_0^1\left((1-s)U_{A^{-1}}(B)+sA^{-1}\right)^{-1}d\nu(s)&& (\text{ by Lemma }\ref{3.3}).
\end{align*}
\end{proof}

\end{theorem}

\begin{proposition}\label{p2}
Let $A$, $B$ be positive invertible elements in a unital JB-algebra $\mathcal{A}$ and $\alpha$ and $\beta$ two nonnegative real numbers. Then\\
i) $A\sharp_rB=\{B(A\sharp_{2-r}B)^{-1}B\}$ for any $r\in(1,2)$\\
ii) $A\sharp_rB=\{A(A\sharp_{-r}B)^{-1}A\}$ for any $r\in(-1,0)$\\
iii) $(A\sharp_r B)=(B\sharp_{1-r}A)$ for any $r\in(-1,2)$\\
iv) $(A\sharp_r B)^{-1}=(A^{-1}\sharp_{r}B^{-1})$ for any $r\in(-1,2)$\\
v) $(\alpha A\sharp_r \beta B)= (\alpha\sharp_r\beta) ( A\sharp_rB)$ for any $r\in(-1,2)$\\
vi) $A\sharp_rB$ is separately operator convex with respect to $A$, $B$, for any $r\in(-1,2)$\\
vii) $\{C(A\sharp_r B)C\}=\{CAC\}\sharp_r\{CBC\}$ for any invertible $C$ in $\mathcal{A}$.
\end{proposition}
\begin{proof}
i)If $r\in(1,2)$, then $2-r\in(0,1)$ and
\begin{align*}
&\{B(A\sharp_{2-r} B)^{-1}B\}= U_B(A\sharp_{2-r} B)^{-1}\\
&=U_B(A^{-1}\sharp_{2-r} B^{-1})&&(\text{ by \eqref{02}} )\\
&=U_BU_{A^{-\frac{1}{2}}}\{A^{\frac{1}{2}}B^{-1}A^{\frac{1}{2}}\}^{2-r}&&(\text{ by \eqref{4}} )\\
&=U_BU_{A^{-\frac{1}{2}}}\{A^{-\frac{1}{2}}BA^{-\frac{1}{2}}\}^{r-2}&&(\text{ by Lemma \ref{l2}} )\\
&=U_BU_{A^{-\frac{1}{2}}}\{A^{-\frac{1}{2}}BA^{-\frac{1}{2}}\}^{r}\{A^{-\frac{1}{2}}BA^{-\frac{1}{2}}\}^{-2}&&(\text{ by Lemma \ref{3.3}} )\\
&=U_BU_{A^{-\frac{1}{2}}}U_{\{A^{\frac{1}{2}}B^{-1}A^{\frac{1}{2}}\}}
\{A^{-\frac{1}{2}}BA^{-\frac{1}{2}}\}^{r}&&( \{a^{-1}a^ra^{-1}\}=a^{r-2} )\\
&=U_BU_{A^{-\frac{1}{2}}}U_{A^{\frac{1}{2}}}U_{B^{-1}}U_{A^{\frac{1}{2}}}\{A^{-\frac{1}{2}}BA^{-\frac{1}{2}}\}^{r}&&(U_{\{aba\}}=U_aU_bU_a )\\
&=U_{A^{\frac{1}{2}}}\{A^{-\frac{1}{2}}BA^{-\frac{1}{2}}\}^{r}&&(\text{ by Lemma \ref{l1}} )\\
&=A\sharp_rB.&&(\text{ by \eqref{4}} )
\end{align*}
ii)
If $r\in(-1,0)$, then $-r\in(0,1)$ and
\begin{align*}
&\{A(A\sharp_{-r} B)^{-1}A\}= U_A(A\sharp_{-r} B)^{-1}\\
&=U_A(A^{-1}\sharp_{-r} B^{-1})&&(\text{ by \eqref{02}} )\\
&=U_A\{A^{-\frac{1}{2}}\{A^{\frac{1}{2}}B^{-1}A^{\frac{1}{2}}\}^{-r}A^{-\frac{1}{2}}\}&&(\text{ by \eqref{4}} )\\
&=U_AU_{A^{-\frac{1}{2}}}\{A^{\frac{1}{2}}B^{-1}A^{\frac{1}{2}}\}^{-r}\\
&=U_{A^{\frac{1}{2}}}\{A^{\frac{1}{2}}B^{-1}A^{\frac{1}{2}}\}^{-r}&&(\text{ by Lemma \ref{3.3} } )\\
&=U_{A^{\frac{1}{2}}}\{A^{-\frac{1}{2}}BA^{-\frac{1}{2}}\}^{r}&&(\text{ by Lemma \ref{l1}} )\\
&=A\sharp_rB.&&(\text{ by \eqref{4}} )
\end{align*}
iii)
If $0\leq r\leq1$, then \eqref{01} shows that $A\sharp_rB=B\sharp_{1-r}A$.
Now, suppose that $r\in(1,2)$, then
\begin{align*}
A\sharp_rB&=\{B(A\sharp_{2-r} B)^{-1}B\}&&(\text{by part(i)})\\
&=\{B(B\sharp_{r-1} A)^{-1}B\}&&(\text{ by \eqref{01}} )\\
&=B\sharp_{1-r}A.&&(\text{by part(ii)})
\end{align*}
If $r\in(-1,0)$, then
\begin{align*}
A\sharp_rB&=\{A(A\sharp_{-r} B)^{-1}A\}&&(\text{by part(ii)}) \\
&=\{A(B\sharp_{1+r} A)^{-1}A\}&&(\text{ by \eqref{01}} )\\
&=B\sharp_{1-r}A.&&(\text{by part(i)})
\end{align*}
iv)\begin{align*}
(A\sharp_rB)^{-1}&=\{A^{\frac{1}{2}}\{A^{-\frac{1}{2}}BA^{-\frac{1}{2}}\}^rA^{\frac{1}{2}}\}^{-1}\\
&=\{A^{-\frac{1}{2}}\{A^{\frac{1}{2}}B^{-1}A^{\frac{1}{2}}\}^rA^{-\frac{1}{2}}\}\\
&=A^{-1}\sharp_rB^{-1}.
\end{align*}

v) it follows directly from the definition.\\
vi)For any $0\leq \lambda\leq 1$, Operator convexity the function $x\rightarrow x^r$ on $\mathcal{A}$ for $r\in[-1,0]\cup[1,2]$ implies that
\begin{align*}
A\sharp_r[(1-\lambda)B_1+\lambda B_2]
&=U_{A^{\frac{1}{2}}}\left((1-\lambda)\{A^{-\frac{1}{2}}B_1A^{-\frac{1}{2}}\}+\lambda\{A^{-\frac{1}{2}}B_2A^{-\frac{1}{2}}\}\right)^r\\
&\leq U_{A^{\frac{1}{2}}}\left((1-\lambda)\{A^{-\frac{1}{2}}B_1A^{-\frac{1}{2}}\}^r+\lambda\{A^{-\frac{1}{2}}B_2A^{-\frac{1}{2}}\}^r\right)\\
&=(1-\lambda)U_{A^{\frac{1}{2}}}\{A^{-\frac{1}{2}}B_1A^{-\frac{1}{2}}\}^r+\lambda U_{A^{\frac{1}{2}}}\{A^{-\frac{1}{2}}B_2A^{-\frac{1}{2}}\}^r\\
&=(1-\lambda)(A\sharp_r B_1)+\lambda(A\sharp_r B_2).
\end{align*}
Similarly $B\sharp_{1-r}A$ is operator convex with respect to $A$.
By Proposition $\ref{p2}$ we have $A\sharp_r B=B\sharp_{1-r}A$, hence $A\sharp_r B$ is also operator convex with respect to $A$.\\
vii) If $0\leq r\leq1$, then \eqref{04} shows that $\{CAC\}\sharp_r\{CBC\}=\{C(A\sharp_rB)C\}$.

Now, suppose that $r\in (1,2)$. By Theorem \ref{t1}, we have
\begin{align*}
&\{CAC\}\sharp_r\{CBC\}\\
&=\int_0^1\left((1-s)\{CBC\}^{-1}+s\left\{\{CBC\}^{-1}\{CAC\}\{CBC\}^{-1}\right\}\right)^{-1}d\mu(s)\\
&=\int_0^1\left((1-s)\{CBC\}^{-1}+sU_{\{CBC\}^{-1}}U_{C}(A)\right)
^{-1}d\mu(s)\\
&=\int_0^1\left((1-s)\{CBC\}^{-1}+sU_{C^{-1}}U_{B^{-1}}U_{C^{-1}}U_{C}(A)\right)
^{-1}d\mu(s)\\
&=\int_0^1\left((1-s)\{CBC\}^{-1}+s\{C^{-1}\{B^{-1}AB^{-1}\}C^{-1}\}\right)^{-1}d\mu(s)\\
&=\int_0^1\left((1-s)\{C^{-1}B^{-1}C^{-1}\}+s\{C^{-1}\{B^{-1}AB^{-1}\}C^{-1}\}\right)^{-1}d\mu(s)\\
&=\int_0^1\{C\left((1-s)B^{-1}+s\{B^{-1}AB^{-1}\}\right)^{-1}C\}d\mu(s)\\
&=\left\{\int_0^1C\left((1-s)B^{-1}+s\{B^{-1}AB^{-1}\}\right)^{-1}d\mu(s)C\right\}\\
&=\{C(A\sharp_rB)C\}.
\end{align*}
If $r\in (-1,0)$, then $\{C(A\sharp_r B)C\}=\{CAC\}\sharp_r\{CBC\}$ by Theorem \ref{t2}.
\end{proof}
In the next theorem, we present a converse of theorem 3 in \cite{wang} for $JB$-algebras.
\begin{theorem}\label{t3}
(Reverse Young inequality for $JB$-algebras) Let $A$, $B$ be positive invertible elements in a unital $JB$-algebra $\mathcal{A}$ and $r\in(-1,0)\cup(1,2)$. Then
\begin{equation}\label{15}
A\nabla_r B\leq A\sharp_rB\leq A!_rB.
\end{equation}
\end{theorem}
\begin{proof}
According to Lemma 2.1 in \cite{bak}, if $r\notin(0,1)$,
\begin{align}\label{12}
(1-r)1+rx\leq x^r\leq((1-r)1+rx^{-1})^{-1}
\end{align}
for all $x>0$.
By functional calculus at $\{A^{-\frac{1}{2}}BA^{-\frac{1}{2}}\}$ for $JB$-algebras [1, proposition 1.21], from \eqref{12} we get
\begin{align*}
(1-r)1+r\{A^{-\frac{1}{2}}BA^{-\frac{1}{2}}\}\leq \{A^{-\frac{1}{2}}BA^{-\frac{1}{2}}\}^r\leq\left(1-r)1+r\{A^{-\frac{1}{2}}BA^{-\frac{1}{2}}\}^{-1}\right)^{-1}.
\end{align*}
Since $U_A$ is a linear mapping, so
\begin{align}\label{13}
&U_{A^\frac{1}{2}}\left((1-r)1+r\{A^{-\frac{1}{2}}BA^{-\frac{1}{2}}\}\right)\leq U_{A^\frac{1}{2}}\left(\{A^{-\frac{1}{2}}BA^{-\frac{1}{2}}\}^r\right)\notag\\
&\leq U_{A^\frac{1}{2}}\left(\left(1-r)1+r\{A^{-\frac{1}{2}}BA^{-\frac{1}{2}}\}^{-1}\right)^{-1}\right).
\end{align}
Inequality \eqref{15} follows from \eqref{13}.
\end{proof}
\section{Heinz and Heron means in JB-algebras}\vspace{.2cm} \noindent

For two positive invertible elements $A$, $B$ in a unital $JB$-algebra $\mathcal{A}$ and $\nu\in[0,1]$,
 we introduce Heinz, Heron and logarithmic means, as follows
\begin{align*}
H_{\nu}(A,B)&=\frac{A\sharp_\nu B+A\sharp_{1-\nu}B}{2}\\
F_{\nu}(A,B)&=(1-\nu)A\sharp_\nu B+\nu A\nabla B\\
L(A,B)=&\int_0^1 A\sharp_\nu B d\nu
\end{align*}
where $A\sharp_\nu B=\{A^{1/2}\{A^{-1/2}BA^{-1/2}\}^{\nu}A^{1/2}\}$.
In \cite[Lemma 1]{bah}, Bhatia gave the following inequality between the Heinz and Heron means,
\begin{align}\label{2.00}
H_\nu(a,b)\leq F_{\alpha(\nu)}(a,b),
\end{align}
for $0\leq\nu\leq1$, where $\alpha(\nu)=(1-2\nu)^2$.

By the proof of Lemma 1 of \cite{bah}, we obtain
\begin{align}\label{2.0}
H_\nu(a,b)\geq F_{\alpha(\nu)}(a,b),
\end{align}
where $\nu\notin (0,1)$.
\begin{theorem}\label{t4}
Let $A$, $B$ be positive invertible elements in a unital JB-algebra $\mathcal{A}$. Then
\begin{align}\label{2.1}
H_{\nu}(A,B)&\leq F_{\alpha(\nu)}(A,B)&& \nu\in[0,1]\\
H_{\nu}(A,B)&\geq F_{\alpha(\nu)}(A,B)&&\nu\notin(0,1),\label{2.1a}
\end{align}
where $\alpha(\nu)=(1-2\nu)^2$.
\end{theorem}

\begin{proof}
Let $\nu\in[0,1]$, from \eqref{2.00} for $x>0$, we get
\begin{align}\label{2.2}
\frac{x^\nu+x^{1-\nu}}{2}\leq (1-\alpha(\nu))x^\frac{1}{2}+\alpha(\nu)\frac{x+1}{2}.
\end{align}
By functional calculus at $\{A^{-\frac{1}{2}}BA^{-\frac{1}{2}}\}$ for $JB$-algebras [1, proposition 1.21], from \eqref{2.2} we get
\begin{align}\label{2.3}
&\frac{\{A^{-\frac{1}{2}}BA^{-\frac{1}{2}}\}^\nu +\{A^{-\frac{1}{2}}BA^{-\frac{1}{2}}\}^{1-\nu}}{2}\notag\\
&\leq (1-\alpha(\nu))\{A^{-\frac{1}{2}}BA^{-\frac{1}{2}}\}^\frac{1}{2}+\alpha(\nu)\frac{\{A^{-\frac{1}{2}}BA^{-\frac{1}{2}}\}+I}{2}.
\end{align}
Since $U_A$ is a linear mapping, so
\begin{align}\label{2.4}
&\frac{U_{A^\frac{1}{2}}(\{A^{-\frac{1}{2}}BA^{-\frac{1}{2}}\}^\nu) +U_{A^\frac{1}{2}}(\{A^{-\frac{1}{2}}BA^{-\frac{1}{2}}\}^{1-\nu})}{2}\notag\\
&\leq (1-\alpha(\nu))U_{A^\frac{1}{2}}(\{A^{-\frac{1}{2}}BA^{-\frac{1}{2}}\}^\frac{1}{2})+\alpha(\nu)\frac{U_{A^\frac{1}{2}}(\{A^{-\frac{1}{2}}BA^{-\frac{1}{2}}\})+A}{2}.
\end{align}
In this case, we get $H_{\nu}(A,B)\leq F_{\alpha(\nu)}(A,B)$. Similarly, for $\nu\notin(0,1)$, inequality \eqref{2.1a} follows from \eqref{2.0}.
\end{proof}

 From \eqref{2.1}, we deduced
\begin{align*}
L(A,B)&=\int_0^1\{A^{\frac{1}{2}}\left(A^{-\frac{1}{2}}BA^{-\frac{1}{2}}\right)^rA^{\frac{1}{2}}\}dt\\
&=\int_0^1H_t(A,B)dt
\leq\int_0^1 F_{\alpha(t)}(A,B)dt\\
&=\int_0^14(t-t^2)(A\sharp B)+(1-4(t-t^2))(\frac{A+B}{2})dt\\
&=\frac{2}{3}(A\sharp B)+\frac{1}{3}\left(\frac{A+B}{2}\right)\\
&=F_{\frac{1}{3}}(A,B).
\end{align*}
It is clear that if $\nu_1\leq\nu_2$ then $F_{\nu_1}(A,B)\leq F_{\nu_2}(A,B)$. Therefore,
\begin{align}\label{2.5}
L(A,B)\leq F_{\beta}(A,B)\leq H_\gamma(A,B),
\end{align}
where $\frac{1}{3}\leq\beta\leq1\leq \gamma$.

\end{document}